\documentclass[11pt]{article}
\usepackage[english]{babel}
\usepackage{amsmath,amsfonts,amssymb,amsthm}
\usepackage{latexsym}
\usepackage{makeidx}
\usepackage{color}
\usepackage{tikz}
\usepackage{vmargin}
\usepackage{enumerate}
\setmarginsrb{20mm}{20mm}{20mm}{20mm}%
            {5mm}{5mm}{5mm}{10mm}

\numberwithin{equation}{section}

\numberwithin{table}{section}

\usepackage{caption}
\usepackage[all]{xy}
\newtheorem{theorem}{Theorem}[section]

\newtheorem{corollary}{Corollary}
\newtheorem*{thm}{Theorem}
\newtheorem*{lemma}{Lemma}
\theoremstyle{definition}

\theoremstyle{remark}
\newtheorem{remark}[theorem]{Remark}

\newcommand{\Z}{\mathbb Z}
\newcommand{\N}{\mathbb N}

\newcommand{\F}{\mathbb F}

\renewcommand{\wr}{\operatorname{\,wr\,}}

\newcommand{\ient}{\widetilde{\textup{ent}}}
\newcommand{\ientd}{\widetilde{\textup{ent}}_{\dim}}

\begin{document}

   \title{\huge A property of the lamplighter group}
   \author{I.~Castellano \and
   G.~Corob Cook \and
   P.H.~Kropholler}
 
   \newcommand{\Addresses}{{
  \bigskip
  \footnotesize

I.~Castellano, \textsc{Dipartimento di Scienze Matematiche, Informatiche e Fisiche - University of Udine, Viale delle Scienze 206, Udine 33100 (Italy)} \textit{E-mail:} \texttt{ilaria.castellano88@gmail.com}

  \medskip

  G.~Corob Cook, \textsc{Ikerbasque - Basque Foundation for Science and Matematika Saila, UPV/EHU, Sarriena s/n, 48940, Leioa- Bizkaia (Spain)}
  \textit{E-mail:} \texttt{gcorobcook@gmail.com}

  \medskip

P.H.~Kropholler, \textsc{Mathematical Sciences - University of Southampton, Southampton SO17 1BJ (United Kingdom}
  \textit{E-mail:} \texttt{p.h.kropholler@soton.ac.uk}

}}
\date{}
\providecommand{\keywords}[1]{\textbf{\textit{Keywords:}} #1}

\maketitle

\begin{abstract} 
{We show that the inert subgroups of the lamplighter group fall into exactly five commensurability classes. 
The result is then connected with the theory of totally disconnected locally compact groups and with algebraic entropy. }
 \end{abstract}
 

A subgroup $H$ of a group $G$ is said to be \emph{inert} if $H$ and $g^{-1}Hg$ are \emph{commensurate} for all $g\in G$, meaning that $H\cap H^g$  always has finite index in both $H$ and $H^g$. The terminology  was introduced by Kegel and has been explored in many contexts (see, for example, the recent survey \cite{dardano_et_al2}). In abstract group theory, Robinson's investigation \cite{Robinson} focusses on soluble groups. Here we study a particular special family of soluble groups: \emph{the lamplighter groups} and our interest is in the connection with the theory of totally disconnected locally compact groups. In that context, inert subgroups are particularly important in the light of van Dantzig's theorem that every totally disconnected locally compact group has a compact open subgroup and of course all such subgroups are commensurate with one another and therefore inert. It should be noted that in recent literature it is common to use the term \emph{commensurated} in place of \emph{inert}, see for example \cite{MR3008315,MR3302969,MR3022552,MR3900762}.   In \S2 we remark a dynamical aspect of the property investigated here and relate it to the concept of  algebraic entropy.

The relation of commensurability is an equivalence relation amongst the subgroups of a group. By a \emph{class} we shall here mean an equivalence class of subgroups under this relation. For a prime $p$, the corresponding lamplighter group is the standard restricted wreath product $\F_p\wr\Z$, i.e., the standard restricted wreath product of a group of order $p$ by an infinite cyclic group. These are the simplest of soluble groups that fall outside the classes considered by Robinson \cite{Robinson}. Our main observation is as follows.

\begin{thm} 
The inert subgroups of the lamplighter group $\F_p\wr\Z$ fall into exactly five classes.
\end{thm}

\section{Proof and application to locally compact groups}

Let $G$ be a group, $K$ a field and $KG$ the group algebra. Let $V$ be a $KG$-module. We say that a $K$-subspace $U$ of $V$ is \emph{$G$-almost invariant} when $U/U\cap Ug$ is finite dimensional for all $g\in G$. We say that subspaces $U$ and $W$ are \emph{almost equal} when $U/U\cap W$ and $W/U\cap W$ are both finite dimensional.

\begin{lemma}\label{lem:main}
Let $G=\langle x\rangle$ be infinite cyclic and $B=KG$. Let $B^+=K[x]$ and $B^-=K[x^{-1}]$. If $A$ is a $G$-almost invariant subspace of $B$ then $A$ is almost equal to one of the four subspaces $0$, $B^+$, $B^-$, $B$.
\end{lemma}

\begin{proof}The group algebra $KG$ is a Laurent polynomial ring and each of its nonzero elements has a \emph{lower degree} and an \emph{upper degree} these being the least integer and the greatest integer for which the corresponding power of $x$ has non-zero coefficient.

Since $A$ is $G$-almost invariant, there is a finite dimensional subspace $F$ of $B$ such that $Ax$ and $Ax^{-1}$ are both contained in $A+F$. Let $n_+$ and $n_-$ denote the maximum and minimum integers in the finite set 
$$\{j;\ x^j\textrm{ belongs to the support of some non-zero element of }F\}.$$ 
We now distinguish two cases each of which has two subcases.
\begin{description}
\item[Case 1] $A$ has an element $a$ with upper degree $\ge n_+$.
\end{description}
By enlarging $F$ if necessary, we may assume $a$ has upper degree $n_+$. We will define a sequence $(a_j)_{j\ge0}$ of elements of $A$ inductively, starting with $a_0=a$, so that $a_j$ has upper degree $n_++j$ and so that for almost all $j$, $a_j$ has lower degree $\ge n_-$.

For $j>0$, we suppose that $a_{j-1}\in A$ has been chosen with upper degree $n_++j-1$. Then $a_{j-1}x$ has upper degree $n_++j$ and $a_{j-1}x=a_j+f$ for some $a_j$ in $A$ and $f$ in $F$. Since $n_++j>n_+$, the upper degree of $a_j$ is the same as that of $a_{j-1}x$, namely $n_++j$. The lower degree of $a_{j-1}x$ is $1$ greater than the lower degree of $a_{j-1}$ and so the lower degree of $a_j$ is either greater than that of $a_{j-1}$ or is $\ge n_-$. As a consequence the terms of the sequence $(a_j)$ eventually all have lower degree $\ge n_-$. Now the span of the $a_j$ is almost equal to $B^+$. 

\begin{description}
\item[Subcase 1a] $A$ has an element with lower degree $\le n_-$.
\end{description}

If this happens then the same reasoning as above produces a sequence with lower degrees decreasing by one and the terms of the sequence span a subspace almost equal to $B^-$. It follows that $A$ is almost equal to $B$.

\begin{description}
\item[Subcase 1b] All elements of $A$ have lower degree $> n_-$.
\end{description}

In this case, $A$ is almost equal to $B^+$.

\begin{description}
\item[Case 2] All elements of $A$ have upper degree $< n_+$.
\end{description}

Similar reasoning shows that either $A$ is finite dimensional or it is almost equal to $B^-$.
\end{proof}

\begin{remark}
Essentially the same strategy can be used to prove a more general result: suppose $R$ is a commutative noetherian ring, $G=\langle x\rangle$ is infinite cyclic, and $B=RG = \bigoplus_{n \in \Z} R$. If $A$ is a $G$-almost invariant $R$-submodule of $B$ then $A$ is almost equal to $\bigoplus_{n < 0} I \oplus \bigoplus_{n>0} J$, for some right ideals $I,J$ of $R$.
\end{remark}

\begin{proof}[Proof of the Theorem]
Let $B$ denote the base of the lamplighter group, i.e., $B$ is the infinite direct sum $\bigoplus_\Z \F_p$ of countably many copies of $\F_p$. This can be identified with the Laurent polynomial ring $\F_p[x,x^{-1}]$. If $H$ is an inert subgroup of $G$ then $H\cap B$ is an $\langle x\rangle$-almost invariant $\F_p$-subspace of $B$ and so is almost equal to one of $0$, $B^+$, $B^-$, $B$ by the above lemma. Commensuration and almost equality are the same thing here because the ground field is finite.

If $H$ has no elements of infinite order then $H\subseteq B$ and we are done. If $H$ has an element of infinite order and also an element of finite order then $H\cap B$ contains Laurent polynomials of arbitrarily large positive and arbitrarily large negative degrees. In this case $H\cap B$ has finite index in $B$ and $H$ has finite index in $G$. If $H$ has an element of infinite order and no elements of finite order then it is infinite cyclic and it is not commensurated.
\end{proof}
 
We thank Pierre-Emmanuel Caprace for pointing out that one now has the following consequence.

\begin{corollary}\label{cor}
If $G$ is a totally disconnected locally compact group which has a dense subgroup isomorphic to a lamplighter group then $G$ is isomorphic to one of the following.
\begin{enumerate}
\item A discrete lamplighter group.
\item
A compact group.
\item
The group $\F_p((t))\rtimes_t\Z$ for some prime $p$.
\item
The unrestricted wreath product $\F_p\overline\wr\Z$ for some prime $p$.
\end{enumerate}
\end{corollary}

In order to justify Corollary~\ref{cor} we need a few preliminaries. Recall that a locally compact group $G$ is  totally disconnected if the identity $1_{_G}$ is its own connected component. For a totally disconnected locally compact group $G$, van Dantzig's theorem ensures that the family of all compact open subgroups of $G$ forms a base of neighbourhoods of $1_{_G}$. Therefore, every totally disconnected locally compact group $G$ has a distinguished class $\mathcal{CO}(G)$ of inert subgroups, namely its compact open subgroups.

For a discrete group $H$, let $\phi\colon H\to G$ be a group homomorphism with dense image in $G$. Such a homomorphism is referred to as a {\em totally disconnected locally compact completion of $H$} (a general framework for totally disconnected locally compact completions can be found in \cite{ReidWes}). Note that $\phi$ is not required to be injective. For every $U\in\mathcal{CO}(G)$, the preimage $\phi^{-1}(U)$  is inert in $H$. This observation connects  the inert subgroups of $H$ to its totally disconnected locally compact completions: given a completion $\phi\colon H\to G$, the set $\{\phi^{-1}(U)\mid U\in\mathcal{CO}(G)\}$ is contained in the commensurability class of an inert subgroup of $H$. Moreover, if we start with an inert subgroup $I$ of $H$, there are two canonical completions of $H$ such that $I$ is the preimage of a compact open subgroup of $G$: the {\em Belyaev completion} \cite{Bel} and the {\em Schlichting completion} \cite{Sch}. By \cite[Theorem~5.4]{ReidWes}, every totally disconnected locally compact completion $\phi$ of the pair $(H,I)$ arises as a quotient (with compact kernel) of the Belyaev completion of $H$ with respect to the inert subgroup $I$. Thus the problem of classifying all totally disconnected locally compact completions of $H$ can be broken into two steps:
\begin{enumerate}
\item classify inert subgroups of $H$ up to commensurability;
\item for each class form the Belyaev completion and classify its quotients with compact kernels.
\end{enumerate}

If we start with the lamplighter group, which is residually finite,  $\F_p\wr\Z$ densely embeds in its Belyaev completions; see \cite[Theorem~7.1]{Bel}. Therefore,  to obtain the exhaustive list in Corollary~\ref{cor} it suffices to form the Belyaev completion of each pair $(H,I)$ where $I$ represents one of the 5 classes of inert subgroups of $\F_p\wr\Z$. 

\section{Connection with algebraic entropy}
The concept of inert subgroup tacitly involves inner automorphisms and, therefore, it is amenable to being extended to the case of a general endomorphism $\varphi$ of a group $G$: a subgroup $H$ of $G$ is said to be {\em $\varphi$-inert} if $H^\varphi\cap H$ has finite index in the image $H^\varphi$ (see \cite{DGBV}). Consequently, a subgroup $H$ is inert in $G$ if $H$ is $\varphi$-inert for every inner automorphism $\varphi$ of $G$. The family of all $\varphi$-inert subgroups of $G$ is denoted by $\mathcal{I}_{\varphi}(G)$.

This definition can be easily adapted to the context of vector spaces:
 for an endomorphism $\phi$ of a $K$-vector space $V$,  a $K$-subspace $U$ of $V$ is  ({\em linearly}) {\em $\phi$-inert} if $\dim_K((U+\phi U)/U)<\infty$. Let $\mathcal{LI}_\phi(V)$ denote the family of all $\phi$-inert linear subspaces of $V$. Notice that $\mathcal{LI}_\phi(V)\subseteq\mathcal{I}_\phi(V)$ whenever $K$ is a finite field.

\medskip

The notion of $\phi$-inert subspace allows to point out a dynamical aspect of the Lemma above. Indeed, let 
\begin{equation}\label{eq:bernoulli}
\beta_K\colon\bigoplus_{n\in\Z} K\rightarrow\bigoplus_{n\in\Z} K,\quad (x_n)_{n\in\Z}\mapsto (x_{n-1})_{n\in\Z}.
\end{equation}
be the {\it two sided Bernoulli shift on $K$}. Then we have the following reformulation of the Lemma.
\begin{corollary}\label{cor2}
One has $\mathcal{LI}_{\beta_K}(V)\cap\mathcal{LI}_{\beta_K^{-1}}(V)=\{0,V^-,V^+,V\}$, where $$V= \bigoplus_{n\in\Z} K,\quad V^+=\bigoplus_{n\geq0} K\quad \text{and}\quad V^-=\bigoplus_{n\leq0} K.$$ 
\end{corollary}
 Several different notions of algebraic entropy have been introduced in the past (see \cite{AKM,Peters,Virili12,SVV} and references there). In particular, the possibility to define $\phi$-inert subobjects has recently turned out to be a very helpful tool for the study of the dynamical properties of the given endomorphism $\phi$. The leading example is the so-called  {\em intrinsic entropy} $\ient$. It was introduced in \cite{DGBV} to obtain a dynamical invariant able to treat also endomorphisms of torsion-free
abelian groups where other entropy functions vanish completely for the lack of non-trivial  finite subgroups. Afterwards, the {\em intrinsic valuation entropy} $\ient_v$ was introduced in \cite{SV} with the aim of extending  $\ient$ to the context of modules over a non-discrete valuation domain and also the algebraic entropy for locally linearly compact vector spaces defined in \cite{CGB} has the same ``intrinsic'' flavour. Therefore, going down the same path, one defines the {\it intrinsic dimension entropy} $\ientd$ for linear endomorphisms by 
\begin{equation}
\ientd(\phi)=\sup_{U\in\mathcal{I}_\phi(V)}\ientd(\phi,U),
\end{equation}
where $\ientd(\phi,U):=\lim_{n\in\N}\frac{1}{n}\dim_K\left(\frac{T_n(\phi,U)}{U}\right)$ and 
$T_n(\phi,U):=U+\phi U+\ldots+\phi^{n-1}U$ for $n\in\N$ (the existence of the limit is not trivial; an easy adaption of the argument of \cite[Proposition~3.1]{CGB} provides a proof).

In this new context, Corollary~\ref{cor2} can be then used to compute  the intrinsic dimension entropy of the two sided Bernoulli shift $\beta_K$, which turns out to equal 1. Indeed, Corollary~\ref{cor2} and a limit-free formula as in \cite{CGB,SV}  provide $\ientd(\beta_K)=\dim_K(V^-/\beta_K^{-1}(V^-))=1$. 

Quite remarkably, $\phi$-inert subspaces do not enrich the dynamics of linear flows like $\phi$-inert subgroups do in the framework of abelian groups (see \cite{DGBV}).
Indeed, one verifies that $\ientd(\phi)=\textup{ent}_{\textup{dim}}(\phi)$ for every $\phi\colon V\to V$, where
$$\textup{ent}_{\textup{dim}}(\phi):=\sup\left\{\lim_{n\to\infty}\frac{\dim_K(T_n(\phi,F))}{n}\mid F\leq V\ \text{and}\ \dim_K(F)<\infty\right\},$$
which is a classical  entropy function for vector spaces and their endomorphisms (details about this entropy function can be found in \cite{GBS}). Indeed, since every finite-dimensional subspace is $\phi$-inert, one easily has $\ientd(\phi)\geq\textup{ent}_{\textup{dim}}(\phi)$. Conversely,
proceeding as in \cite[Lemma 3.9]{CGB}, for every $U\in\mathcal{I}_\phi(V)$ one can find a finite-dimensional subspace $F_U$ such that $T_n(\phi,U)=U+T_n(\phi,F_U)$. Consequently, $\dim_K(T_n(\phi,U)/U)\leq \dim_K(T_n(\phi,F_U))$ and $\ientd(\phi)\leq \textup{ent}_{\textup{dim}}(\phi)$. 

In other words, there are always enough finite-dimensional linear subspaces.

\section*{Acknowledgement}
I.C. and P.H.K. were partially supported by EPSRC grant EP/N007328/1. G.C.C. was supported by ERC grant 336983 and the Basque government grant IT974-16.

\bibliography{CCCKBIB}
\bibliographystyle{abbrv}

\Addresses

\end{document}